\documentclass{amsart}

\usepackage{enumerate}
\usepackage{amssymb, amsmath}
\usepackage{mathrsfs}
\usepackage{amscd}
\usepackage[active]{srcltx}
\usepackage{verbatim}

\def\R{{\mathbb R}}
\def\C{{\mathbb C}}

\newcommand{\E}{{\mathbb E}}

\newcommand{\F}{{\mathscr F}}
\renewcommand{\H}{{H_\infty}}

\newcommand{\HH}{\mathscr{H}}

\newcommand{\e}{\varepsilon}

\newcommand{\Om}{\Omega}

\newcommand{\Ito}{{\hbox{\rm It\^o}}}

\renewcommand{\L}{L^2(0,T)}

\newcommand{\beq}{\begin{equation}}
\newcommand{\eeq}{\end{equation}}
\newcommand{\bal}{\begin{aligned}}
\newcommand{\eal}{\end{aligned}}
\newcommand{\ben}{\begin{enumerate}}
\newcommand{\beni} {\begin{enumerate}[(i)]}
\newcommand{\een}{\end{enumerate}}
\newcommand{\bit}{\begin{itemize}}
\newcommand{\eit}{\end{itemize}}
\newcommand{\beqw}{\begin{equation*}}
\newcommand{\eeqw}{\end{equation*}}
\newcommand{\bthm}{\begin{thm}}
\newcommand{\ethm}{\end{thm}}
\newcommand{\bpr}{\begin{prop}}
\newcommand{\epr}{\end{prop}}
\newcommand{\ble}{\begin{lem}}
\newcommand{\ele}{\end{lem}}
\newcommand{\blem}{\begin{lem}}
\newcommand{\elem}{\end{lem}}
\newcommand{\bpf}{\begin{proof}}
\newcommand{\epf}{\end{proof}}
\newcommand{\bex}{\begin{ex}}
\newcommand{\eex}{\end{ex}}
\newcommand{\bre}{\begin{ex}}
\newcommand{\ere}{\end{ex}}
\renewcommand{\le}{\leqslant}
\renewcommand{\leq}{\leqslant}
\renewcommand{\ge}{\geqslant}
\renewcommand{\ge}{\geqslant}
\newcommand{\bma}{\begin{bmatrix}}
\newcommand{\ema}{\end{bmatrix}}

\renewcommand{\Re}{\hbox{\rm Re}}

\newcommand{\Dom}{{\mathsf D}}
\newcommand{\D}{{\mathsf D}_p}

\newcommand{\wt}{\widetilde}

\newcommand{\calL}{{\mathscr L}}

\newcommand{\n}{\Vert}
\newcommand{\g}{\gamma}

\newcommand{\embed}{\hookrightarrow}
\newcommand{\s}{^*}
\newcommand{\lb}{\langle}
\newcommand{\rb}{\rangle}

\renewcommand{\SS}{{\bf S}}

\newcommand{\ip}[1]{\langle {#1}\rangle}

\newcommand{\Ran}{\mathsf{R}}

\renewcommand{\Re}{{\rm{Re}}\;}

\renewcommand{\i}{_\infty}

\newcommand{\ot}{\otimes}

\newcommand{\uL}{{L}}

\newcommand{\uD}{D}

\renewcommand{\L}{\mathscr{L}}

 \newtheorem{thm}{Theorem}[section]
 \newtheorem{cor}[thm]{Corollary}
 \newtheorem{lem}[thm]{Lemma}
 \newtheorem{prop}[thm]{Proposition}
 \theoremstyle{definition}
 
 \theoremstyle{remark}
 \newtheorem{rem}[thm]{Remark}
 \newtheorem{ex}{Example} 
 \numberwithin{equation}{section}

\begin{document}

\title[Ornstein-Uhlenbeck operators]
{Gradient estimates and domain identification for analytic Ornstein-Uhlenbeck operators}

\author{Jan Maas}
\address{
Institute for Applied Mathematics\\
University of Bonn\\
Endenicher Allee 60\\
53115 Bonn\\
Germany} 
\email{maas@iam.uni-bonn.de}

\author{Jan van Neerven}
\address{Delft Institute of Applied Mathematics\\
Delft University of Technology\\ P.O. Box 5031\\ 2600 GA Delft\\The
Netherlands}
\email{J.M.A.M.vanNeerven@TUDelft.nl}

\thanks{The authors are supported by VIDI subsidy 639.032.201
(JM) and VICI subsidy 639.033.604 (JvN)
of the Netherlands Organisation for Scientific Research (NWO)}

\dedicatory{Dedicated to Herbert Amann on the occasion of his 70th birthday}



\begin{abstract}
Let $P$ be the Ornstein-Uhlenbeck semigroup associated
with the stochastic Cauchy problem
$$ dU(t) = AU(t)\,dt + dW_H(t), $$
where $A$ is the generator of a $C_0$-semigroup $S$ on a Banach space $E$,
$H$ is a Hilbert subspace of $E$,
and $W_H$ is an $H$-cylindrical Brownian motion. Assuming that
$S$ restricts to a $C_0$-semigroup on $H$, we obtain $L^p$-bounds for
$D_H P(t)$.
We show that if $P$ is analytic, then the invariance assumption
is fulfilled. As an application we determine the
$L^p$-domain of the generator of $P$ explicitly in the case where
$S$ restricts to a $C_0$-semigroup on $H$ which is similar to an analytic
contraction semigroup. The results are applied to the 1D stochastic heat equation
driven by additive space-time white noise.
\end{abstract}

\maketitle

\section{Introduction}

Consider the stochastic  Cauchy problem
\beq\tag{SCP}\label{SCP}
\bal dU(t) & = AU(t)\,dt + dW_H(t), \quad t\ge 0, \\
 U(0) & = x.
 \eal
\eeq Here $A$ generates a $C_0$-semigroup $S = (S(t))_{t\ge 0}$ on a real Banach space $E$,
$H$ is a real Hilbert subspace continuously embedded in $E$, $W_H$
is an $H$-cylindrical Brownian motion on a probability space $(\Om,\F\,P)$,
and $x\in E$. A {\em weak solution} is a measurable adapted
$E$-valued process $U^x = (U^x(t))_{t\ge 0}$ such that
$t\mapsto U^x(t)$ is integrable almost surely and for all $t\ge 0$ and $x\s\in \Dom(A\s)$ one has
$$ \ip{U^x(t),x^*} = \ip{x,x^*} + \int_0^t \lb U^x(s),A\s x\s\rb\,ds + W_H(t)i\s x\s \ \ \hbox{almost surely.}$$
Here $i:H\embed E$ is the inclusion mapping.
A necessary and sufficient condition for the existence of a weak solution is that
the operators $I_t:L^2(0,t;H)\to E$,  $$ I_t g := \int_0^t S(s)i g(s)\,ds,$$
are $\g$-radonifying for all $t\ge 0$. If this is the case, then $s\mapsto S(t-s)i$ is stochastically integrable on $(0,t)$ with respect to $W_H$ and the process $U^x$ is given by
$$ U^x(t) = S(t)x +  \int_0^t S(t-s)i\,dW_H(s), \quad t\ge 0.$$
For more information and an explanation of the terminology we refer to \cite{vNW05}.

Assuming the existence of the solution $U^x$, on the Banach space $C_{\rm b}(E)$ of all bounded continuous functions $f:E\to \R$ one defines the {\em Ornstein-Uhlenbeck semigroup}
$P = (P(t))_{t\ge 0}$ by
\begin{equation}\label{eq:P}
 P(t)f(x):= \E f(U^x(t)), \quad t\ge 0, \ x\in E.
\end{equation}
The operators $P(t)$ are linear contractions on $C_{\rm b}(E)$ and satisfy $P(0)
= I$ and  $P(s)P(t) = P(s+t)$ for all $s,t\ge 0$. For all $f\in C_{\rm b}(E)$ the mapping
$(t,x)\mapsto P(t)f(x)$ is continuous, uniformly on compact subsets of $[0,\infty)\times E$.

If the operator $I_\infty:L^2(0,\infty;H)\to E$ defined by  $$
I_\infty g := \int_0^\infty S(t)i g(t)\,dt$$ is $\g$-radonifying,
then the problem \eqref{SCP} admits a unique invariant measure
$\mu_\infty$. This measure is a centred Gaussian Radon measure on
$E$, and its covariance operator equals $I_\infty I_\infty\s$.
Throughout this paper we shall assume that this measure exists; if
\eqref{SCP} has a solution, then this assumption is for instance
fulfilled if $S$ is uniformly exponentially stable. The reproducing
kernel Hilbert space associated with $\mu_\infty$ is denoted by
$H_\infty$. The inclusion mapping $H_\infty\embed E$ is denoted by
$i_\infty$. Recall that $Q_\infty:= i_\infty i_\infty\s = I_\infty
I_\infty\s$. Is is well-known that $S$ restricts to a
$C_0$-contraction semigroup on $H_\infty$ \cite{CG96} (the proof for
Hilbert spaces $E$ extends without change to Banach spaces $E$),
which we shall denote by $S_\infty$.

By a standard application of Jensen's inequality,
the semigroup $P$ has a unique extension to a $C_0$-contraction semigroup to the
spaces $L^p(E,\mu\i)$, $1\le p<\infty$. By slight abuse of notation we shall denote this semigroup by $P$ again. Its infinitesimal generator will be denoted by $L$. In order to give an explicit expression for $L$ it is useful to introduce, for integers $k,l\ge 0$, the space $\F C_{\rm b}^{k,l}(E)$ consisting of all functions $f\in  C_{\rm b}(E)$ of the form
$$ f(x) = \varphi(\lb x,x_1\s\rb, \dots, \lb x,x_N\s\rb)$$
with $f\in C_{\rm b}^k(\R^N)$ and  $x_1^*,\dots,x_N^*\in
\Dom(A^{*l})$. With this notation one has that $\F C_{\rm
b}^{2,1}(E)$ is a core for $L$, and on this core one has
$$ Lf(x) = \frac12 \,{\rm tr}\, D_H^2 f(x)+ \lb x,A\s Df(x)\rb.$$
Here,
$$\bal D_H f(x) & = \sum_{n=1}^N \frac{\partial \varphi}{\partial x_n}(\lb x,x_1\s\rb, \dots, \lb x,x_N\s\rb)\otimes i\s x_n\s,\\
 D f (x)& = \sum_{n=1}^N \frac{\partial \varphi}{\partial x_n}(\lb x,x_1\s\rb, \dots, \lb x,x_N\s\rb)\otimes x_n\s,
\eal$$
denote the Fr\'echet derivatives into the directions of $H$ and $E$, respectively.

\section{Gradient estimates: the $H$-invariant case}\label{sec:gradient}

Our first result gives a pointwise gradient bound for $P$ under the assumption that
$S$ restricts to a $C_0$-semigroup on $H$ which will be denoted by $S_H$. 
As has been shown in \cite[Corollary 5.6]{GGvN03},
under this assumption the operator $D_H$ is closable as a densely defined
operator from $L^p(E,\mu\i)$ to $L^p(E,\mu;H)$
for all $1\le p<\infty$. The domain of its closure is denoted by $\D(D_H)$.

 \begin{prop}[Pointwise gradient bounds]
 \label{prop:pointwisegrad}
If $S$ restricts to a $C_0$-semigroup on $H$, then for all $1 <
p <\infty$ there exists a constant $C\ge 0$ such that for all $t>0$
and $f\in \F C_{\rm b}^{1,0}(E)$ we have
\begin{align*}
  \sqrt{t} |D_H P(t) f(x)| \le C \kappa(t) (P(t)|f|^p(x))^{1/p},
\end{align*}
where $ \kappa(t) := \sup_{s \in [0,t]} \|S_H(s)\|_{\L(H)}.$
\end{prop}

 \begin{proof}
The proof follows the lines of \cite[Theorem 8.10]{MN-kato} and is
inspired by the proof of \cite[Theorem 6.2.2]{DZ},
where the null controllable case was considered.

The distribution $\mu_t$ of the random variable $U^0(t)$ is a centred
Gaussian Radon measure on $E$. Let $H_t$ denote its RKHS and let
$i_t:H_t\embed E$ be the inclusion mapping. As is well  known and
easy to prove, cf. \cite[Appendix B]{DaPZab} one has
$$ H_t  = \Big\{\int_0^t S(t-s)ig(s)\,ds: \ g\in L^2(0,t;H)\Big\}$$
with
$$ \n h\n_{H_t} = \inf\Big\{ \n g\n_{L^2(0,t;H)}: \ h = \int_0^t S(t-s)ig(s)\,ds\Big\}.$$
The mapping $$\phi^{\mu_t}: i_t\s x\s\mapsto \lb \cdot,x\s \rb, \quad
x\s\in E\s,$$ defines an isometry from $H_t$ onto a closed subspace
of $L^2(E,\mu_t)$. For $h\in H_t$ we shall write
$\phi^{\mu_t}_{h}(x):= (\phi^{\mu_t} h)(x)$.

Fix $h\in H$. Since $S$ restricts to a $C_0$-semigroup $S_H$ on $H$ we may
consider the function $g\in L^2(0,t;H)$ given by $g(s) = \frac1t
S(s)h$. From the identity $S(t)h = \int_0^t S(t-s)g(s)\,ds$ we deduce
that $S(t)h\in H_t$ and \beq\label{est:Ht} \n S(t)h\n_{H_t}^2 \leq \n
g\n_{L^2(0,t;H)}^2 = \frac1{t^2} \int_0^t \n S(s)h\n_H^2\,ds \le
\frac1t \kappa(t)^2 \n h\n_H^2. \eeq

Fix a function $f\in \F C_{\rm b}^{1,0}(E)$, that is, $f(x) =
\varphi(\lb x,x_1\s\rb,\dots,\lb x,x_N\s\rb)$ with $\varphi\in C_{\rm
b}^{1}(\R^N)$ and $x_1\s,\dots,x_N\s\in E\s$. It is easily checked
that for all $t>0$ we have $P(t)f\in \F C_{\rm b}^{1,0}(E)$; in
particular this implies that $P(t)f\in \D(D_H)$. By the
Cameron-Martin formula \cite{Bo},
$$
\bal
 \frac1{\e}\big(P(t)f(x+\e h)-\!P(t)f(x)\big)
 & = \frac1\e \int_E \big(f(S(t)(x + \e h) +y) - f(S(t)x+y)\big)\,d\mu_t(y)
\\ & =  \int_E \frac1\e(E_{\e S(t)h} -1) f(S(t)x+y)\,d\mu_t(y),
\eal
$$
where for $h\in H_t$ we write $$E_{h}(x):= \exp(\phi^{\mu_t}_{h}(x) -
\tfrac12\n h\n_{H_t}^2).$$ It is easy to see that for each $h\in H_t$
the family $\big(\frac1\e(E_{\e h}-1)\big)_{0<\e<1}$ is uniformly
bounded in $L^2(E,\mu_t)$, and therefore uniformly integrable in
$L^1(E,\mu_t)$. Passage to the limit $\e\downarrow 0$ in the previous
identity now gives
$$
[D_H P(t)f(x), h]
 = \int_E  f(S(t)x+y)\phi^{\mu_t}_{S(t)h}(y)\,d\mu_t(y).
$$
By H\"older's inequality with $\frac1r+\frac1q=1$ and the
Kahane-Khintchine inequality, which can be applied since $\phi^{\mu_t}_{S(t)h}$ is a Gaussian random variable,
$$
\bal \ & |[D_H P(t)f(x), h]|
\\ & \qquad \le \Big(\int_E |f(S(t)x+y)|^r\,d\mu_t(y)\Big)^\frac1r
\Big(\int_E |\phi^{\mu_t}_{S(t)h}(y)|^q\,d\mu_t(y)\Big)^\frac1q
\\ & \qquad \le K_q
  \Big(\int_E |f(S(t)x+y)|^r\,d\mu_t(y)\Big)^\frac1r\Big(\int_E
|\phi^{\mu_t}_{S(t)h}(y)|^2\,d\mu_t(y)\Big)^\frac12
\\ & \qquad = K_q  (P(t)|f|^r(x))^\frac{1}{r} \n S(t)h\n_{H_t}.
\eal
$$
Using \eqref{est:Ht}  we find that
$$
\bal \ & \big|\sqrt{t}[D_H P(t)f(x), h]\big|
 \leq K_q \kappa(t)(P(t)|f|^r(x))^{\frac{1}{r}}
 \n h\n_{H}, \eal
$$
and by taking the supremum over all $h\in H$ of norm 1 we obtain the
desired estimate.
 \end{proof}

\begin{cor}
If $S$ restricts to a $C_0$-semigroup on $H$, then
for all $1<p<\infty$ the operators $D_HP(t)$, $t>0$, extend uniquely
to bounded operators from $L^p(E,\mu\i)$ to $L^p(E,\mu\i;H)$,
and there exists a constant
$C\ge 0$ such that for any $t > 0,$
$$ \sqrt{t}\n D_H P(t)\n_{\calL(L^p(E,\mu\i), L^p(E,\mu\i;H))} \le
 C \kappa(t).$$
\end{cor}
\bpf Integrating the inequality of the proposition and using the fact
that $\mu_\infty$ is an invariant measure for $P$ we obtain
$$
\bal
 \n \sqrt{t} D_H P(t)f\n_{L^p(E,\mu\i)}^p
 & \le C^p \kappa(t)^p\int_E P(t)|f|^p(x)\,d\mu\i(x)
 \\ & = C^p\kappa(t)^p \int_E |f|^p(x)\,d\mu\i(x)
= C^p \kappa(t)^p\n f\n_{L^p(E,\mu\i)}^p.
 \eal
 $$
\epf

\section{Gradient estimates: the analytic case}\label{sec:gradient-anal}

Analyticity of the semigroup
$P$ on $L^p(E,\mu_\infty)$ has been investigated by several authors \cite{Fu95, Go99, GvN03, MN1}.
The following result of \cite{GvN03} is our starting point. Recall that in the definition of an {\em analytic $C_0$-contraction semigroup}, contractivity is required on an open sector containing the positive real axis. 

\begin{prop} \label{prop:analytic} For any $1 <p<\infty$ 
the following assertions are equivalent:
\ben
\item[\rm(1)]  $P$ is an analytic $C_0$-semigroup on
$L^p(E,\mu_\infty)$;
\item[\rm(2)] $P$ is an analytic $C_0$-contraction semigroup on $L^p(E,\mu_\infty)$;
\item[\rm(3)] $S$ restricts to an analytic $C_0$-contraction semigroup on
$H_\infty$;
\item[\rm(4)] $Q_\infty A^*$ acts as a bounded operator in
${H}$. \een
\end{prop}

A more precise formulation of \rm(4) is that there should exist a
bounded operator $B:H\to H$ such that
$ i B i\s x\s = Q_\infty  A\s x\s$ for all $x\s\in E\s.$
The identity $Q_\infty A\s + AQ_\infty = -ii\s $ implies that
$B+B^* = -I$.

In what follows we shall simply say that `$P$ is analytic' to express that the equivalent
conditions of the proposition are satisfied for some (and hence for all) $1< p<\infty$.

The next result has been shown in \cite{MN1} for $p=2$ and was extended to
$1<p<\infty$ in \cite{MN-kato}.

\begin{prop}\label{prop:MN} If $P$ is analytic, then
$\F C_{\rm b}^{2,1}(E)$
is a core for the generator $L$ of $P$ in $L^p(E,\mu_\infty)$, and on this core $L$ is given
by
$$L = D_H\s B D_H.$$
\end{prop}

Our first aim is to show that analyticity of $P$ implies that
$H$ is $S$-invariant. For self-adjoint $P$ this was proved in \cite{CG02,GvN03}.

\begin{thm}\label{thm:H-invar} If $P$ is analytic, then $S$ restricts to a bounded analytic
$C_0$-semigroup $S_H$ on $H$.
\end{thm}
\bpf Consider the linear mapping 
\begin{equation}\label{eq:defV} 
V: i\i\s x\s\mapsto i\s x\s,\quad x\s\in E\s.
\end{equation}
It is shown in \cite{GGvN03} that $i\i\s x\s = 0$
implies $i\s x\s =0$, so that this mapping is well-defined, and that
the closability of $D_H$ implies the closability of $V$ as a densely
defined operator from $\H$ to $H$. With slight abuse of notation we
denote its closure by $V$ again and let $\Dom(V)$ the domain of the
closure.

By \cite[Proposition 7.1]{AMN}, the operator $-VV^*B$ is sectorial of
angle $<\frac{\pi}{2}$, and therefore $G := VV^*B$ generates a
bounded analytic $C_0$-semigroup $({T}(t))_{t \geq 0}$ on ${H}.$ To
prove the theorem, by uniqueness of analytic continuation and duality
it suffices to show that ${T(t)}\circ{i^*} = i^*\circ{S^*(t)}$ for all
$t\ge 0$.

For all $x^* \in \Dom(A^*)$ we have $B{i}^*x^* \in
\Dom(V^*)$ and $V^*B{i}^*x^* = i_\infty^* A^*x^*$. Indeed, for $y^*
\in E^*$ one has
 $$
 [B{i}^*x^*,Vi_\infty^*y^*]
 =\ip{i_\infty^*A^*x^*,i_\infty^*y^*},
 $$
which implies the claim. By applying the operator $V$ to this
identity we obtain ${i}^*x^* \in \Dom(G)$ and $G \,{i}^*x^* =
{i}^*A^*x^*$, from which it follows that
${T}(t)i^* x^* = {i}^*S^*(t)x^*$. This proves the theorem, with $S_H = T\s$.\epf

This result should be compared with \cite[Theorem 9.2]{GvN03}, where it is shown that if $S$ restricts to an analytic $C_0$-semigroup on $H$ which is contractive in some equivalent Hilbert space norm, then
$P$ is analytic on $L^p(E,\mu\i)$.

Under the assumption that $P$ is analytic on $L^p(E,\mu\i)$, the gradient estimates of the previous section can be improved as follows.
Recall that a collection of bounded operators $\mathscr{T}$ between Banach spaces $X$ and $Y$ is said to be {\em $R$-bounded} if there exists a constant $C$ such that for any finite subset $T_1, \ldots, T_n \subset \mathscr{T}$ and any $x_1, \ldots, x_n \in X$ we have
\begin{align*}
  \E \Big\| \sum_{j=1}^n r_j T_j x_j \Big\|^2 \leq C^2
  \E \Big\| \sum_{j=1}^n r_j  x_j \Big\|^2,
\end{align*}
where $(r_j)_{j\ge 1}$ is an independent collection of Rademacher random variables. The notion of $R$-boundedness plays an important role in recent advances in the theory of evolution equations (see \cite{DHP,KuW}).

\begin{thm}\label{thm:grad} If $P$ is analytic, then for all $1<p<\infty$
the set $$\{\sqrt{t}D_H P(t): \ t>0\}$$ is $R$-bounded in $\calL(L^p(E,\mu\i), L^p(E,\mu\i;H))$ and we have the square function estimate
$$ \Big\n  \Big(\int_0^t \n D_H P(t)f\n_H^2 \,dt \Big)^{1/2}\Big\n_{L^p(E,\mu\i)} \lesssim \n f\n_{L^p(E,\mu\i)}$$
with implied constant independent of $f\in L^p(E,\mu\i)$.
\end{thm}
\begin{proof}
By Proposition \ref{prop:MN} and Theorem \ref{thm:H-invar}, the theorem is a special case of
\cite[Theorem 2.2] {MN-kato}.
\end{proof}

The above result plays a crucial role in our recent paper \cite{MN-kato} in which  $L^p$-domain characterisations for the operator $L$ and its square root have been obtained. Before stating the result, let us informally sketch how Theorem \ref{thm:grad} enters the argument. 
\renewcommand{\uL}{\underline{L}}
In order to prove a domain characterisation for the operator $L,$ we first aim to obtain two-sided estimates for $\| \sqrt{-L} f \|_{L^p(E,\mu_\infty)}$ in terms of suitable Sobolev norms. For this purpose we consider a variant of an operator theoretic framework introduced in \cite{AKM} in the analysis of the famous Kato square root problem. 
The idea behind this framework is that the second order operator $L$ can be naturally studied through the first order Hodge-Dirac-type operator
\begin{align*}
  \qquad \Pi =  \bma  0 &  -\uD_H^* B  \\ \uD_H & 0  \ema
              \text{ on $L^p(E,\mu_\infty) \oplus L^p(E,\mu_\infty;H)$}.
\end{align*}
This operator is bisectorial and its square is the sectorial operator given by
\begin{align*}
- \Pi^2 =  \bma  \uD_V^* B \uD_V & 0   \\ 0  & \uD_V \uD_V^* B \ema
     =  \bma  L & 0  \\ 0 & \uL  \ema,
\end{align*}
where $\uL := \uD_V \uD_V^* B.$
The approach in \cite{MN-kato} consists of proving estimates for $\sqrt{-L}f$ along the lines of the following formal calculation: 
 \begin{align*}
    \|\uD_H f \|_p
   & = \| \Pi (f,0) \|_p
 \leq \| \Pi/\sqrt{\Pi^2} \|_{p} \,
            \| \sqrt{\Pi^2} (f,0) \|_p
     = \| \Pi/\sqrt{\Pi^2} \|_{p} \,
            \| \sqrt{L}f \|_p.
 \end{align*}
Oversimplifying things considerably, the proof consists of turning this calculation into rigourous mathematics. This can be done once we know that the operator $\Pi/\sqrt{\Pi^2}$ is bounded. Since the function $z \mapsto z /\sqrt{z^2}$ is a bounded analytic function on each bisector around the real axis, it suffices to show that $\Pi$ has a bounded $H^\infty$-functional calculus. This in turn will follow if we show that
\begin{enumerate}
\item the resolvent set $\{(it - \Pi)^{-1} \}_{t\in \R \setminus \{0\}}$ is $R$-bounded;
\item the operator $\Pi^2$ admits a bounded functional calculus.
\end{enumerate}
To prove (1), we observe that 
$$ (I-it \Pi)^{-1} = \bma (1-t^2 L )^{-1} & -it (I-t^2 L )^{-1}\uD_H^* {B}\\
  it \uD_H  (I-t^2 L )^{-1} & (I-t^2  \uL )^{-1}  \ema, \quad
t\in \R\setminus\{0\}.
$$
It suffices to prove $R$-boundedness for each of the entries separately. The diagonal entries can be dealt with using abstract results on $R$-boundedness for positive contraction semigroups on $L^p$-spaces. The $R$-boundedness for the off-diagonal entries can be derived using Theorem \ref{thm:grad}.

To prove (2) we use the fact that the semigroup generated by $\uL$ equals 
$P \ot S_H^*$ on the range of the gradient $D_H.$ Here $S_H$ denotes the restriction 
of the semigroup $S$ to $H$ (see Theorem \ref{thm:H-invar}). Therefore (2) follows, 
provided that the negative generator $-A_H$ of $S_H$ has a bounded $H^\infty$-calculus. 
This reduces the original question about $\sqrt{-L}$ to a question about the operator 
$A_H,$ which is defined 
directly in terms of the data $H$ and $A$ of the problem. The latter question 
should be thought of as expressing the compatibility of the drift 
(represented by the operator $A$) and the noise (represented by the Hilbert space $H$).
This compatibility condition is not automatically satisfied. 
In fact, by a result of Le Merdy \cite{LeM-sim}, $-A_H$ admits a bounded 
$H^\infty$-functional calculus on $H$
if and only if $S_H$ is an analytic $C_0$-{\em contraction} semigroup on
$H$ with respect to some equivalent Hilbert space norm. Such needs not always
be the case, as is shown by well-known examples
\cite{McIYag}.
\renewcommand{\HH}{\mathscr{H}}
\renewcommand{\AA}{\mathscr{A}}
\renewcommand{\SS}{\mathscr{S}}
\newcommand{\LL}{\mathscr{L}}
\newcommand{\PP}{\mathscr{P}}

The following result summarises the informal discussion above and provides an additional equivalent condition in terms of the operator $A_\infty$. In this result we let $\D(D_H^2)$ denote the second order Sobolev space associated with the operator $D_H.$

\begin{thm}\label{thm:kato} Let $1<p<\infty$.
If $P$ is analytic on $L^p(E,\mu\i)$, then the following assertions are equivalent:
\ben
\item[\rm(1)] $\D(\sqrt{-L}) = \D(D_H)$ with norm equivalence
$$\n \sqrt{-L}f\n_{L^p(E,\mu\i)} \eqsim \n D_H
f\n_{L^p(E,\mu\i;H)};$$
\item[\rm(2)] $\Dom(\sqrt{-A_\infty}) = \Dom(V)$ with norm equivalence
  $$\n \sqrt{-A\i}h\n_{\H} \eqsim \n Vh \n_H;$$
\item[\rm(3)] $-A_H$ admits a bounded $H^\infty$-functional calculus on $H$.
\een
If these equivalent conditions are satisfied we have
$$ \D(L) = \D(D_H^2)\cap \D(A_\infty\s D),$$
where $D$ is the Malliavin derivative in the direction of $H_\infty$.
\end{thm}
\begin{proof}
By Proposition \ref{prop:MN}
and Theorem \ref{thm:H-invar}, the theorem is a special case of
\cite[Theorems 2.1, 2.2]{MN-kato} provided we replace $A_\infty$ by $A_\infty\s$
in (2). The equivalence of (2) for $A_\infty$ and $A_\infty\s$, however, is well
known (see also \cite[Lemma 10.2]{MN-kato}).
\end{proof}

The problem of identifying the domains of $\sqrt{-L}$ and $L$ has a long and
interesting history. We finish this paper by presenting three known
special cases of Theorem \ref{thm:kato}. In each case, it is easy to
verify that (3) is satisfied.

\begin{ex}
For the classical Ornstein-Uhlenbeck operator, which corresponds to
$H=E=\R^d$ and $A = -I$, conditions (2) and (3) of Theorem
\ref{thm:kato} are  trivially fulfilled and (1) reduces to the
classical Meyer inequalities of Malliavin calculus. For a discussion
of Meyer's inequalities we refer to the book of Nualart  \cite{Nua}.
\end{ex}

\begin{ex}\label{ex:symm}
Meyer's inequalities were extended to infinite dimensions by
Shige\-kawa \cite{Sh92}, and Chojnowska-Michalik and Goldys
\cite{CG01,CG02}, who considered the case where $E$ is a Hilbert
space and $A_H$ is self-adjoint. Both authors deduce the generalised
Meyer inequalities from square functions estimates. The identification of
$\D(L)$ in the self-adjoint case is due to Chojnowska-Michalik and Goldys
\cite{CG01,CG02}, who extended the case $p=2$ obtained earlier by Da Prato
\cite{DP97}.
\end{ex}

So far, these examples were concerned with the selfadjoint case.

\begin{ex}\label{ex:FD}
A non-selfadjoint extension of Meyer's inequalities has been given
for the case $E=\R^d$ by Metafune, Pr\"uss, Rhandi, and
Schnaubelt \cite{MPRS02} under the non-degeneracy assumption $H=\R^d$. In this
situation the semigroup $P$
is analytic on $L^p(\mu_\infty)$ \cite{Fu95}, see also \cite{Go99,
GvN03}; no symmetry assumptions need to be imposed on $A$. The $S$-invariance of $H$ and the fact that the generator of
$S = S_H$ admits a bounded $H^\infty$-calculus are trivial. Therefore,
(3) is satisfied again. Note that the domain characterisation reduces to $\D(L)
= \D(D^2)$, where $D$ is the derivative on $\R^d$.
The techniques used in \cite{MPRS02} to prove
(1) are very different, involving diagonalisation arguments and the
non-commuting Dore-Venni theorem. The identification of
$\Dom_p(L) = \D(D^2)$ for $p=2$ had been obtained previously by Lunardi
\cite{Lu97}.
\end{ex}

Our final corollary extends the characterisations of $\D(L)$
contained in Examples \ref{ex:symm} and \ref{ex:FD} and lifts the non-degeneracy
assumption on $H$ in Example \ref{ex:FD}.
\begin{cor}
If $S$ restricts to an analytic
$C_0$-semigroup on $H$ which is contractive with respect to some
equivalent Hilbert space norm, then
for all $1<p<\infty$ we have
$$ \D(L) = \D(D_H^2)\cap \D(A_\infty\s D),$$
where $D$ is the Malliavin derivative in the direction of $H_\infty$.
\end{cor}
\begin{proof}
As has already been mentioned in the discussion preceding Theorem \ref{thm:grad},
the assumptions imply that $P$ is analytic.
Moreover, since the restricted semigroup
$S_H$ is similar to an analytic contraction semigroup,
its negative generator $-A_H$
admits a bounded $H^\infty$-calculus, and the result follows from
Theorem \ref{thm:kato}.
\end{proof}

Let us finally mention that the results in \cite{MN-kato} have been proved for a more general class of elliptic operators on Wiener spaces (cf. Section 3 of that paper). In this setting the data consist of
 \begin{itemize}
\item an arbitrary Gaussian measure $\mu$ on a separable Banach space $E$ with reproducing kernel Hilbert space $\mathscr{H}$;
\item an analytic $C_0$-contraction semigroup $\SS$ on $\HH$ with generator $\mathscr{A}.$
\end{itemize}
Given these data, the semigroup $\PP$ is defined on $L^2(E,\mu)$ by second quantisation of the semigroup $\SS$. Roughly speaking, this means that one uses the Wiener-\Ito\ isometry to identify $L^2(E,\mu)$ with the symmetric Fock space over $\HH$, i.e., the direct sum of symmetric tensor powers of $\HH.$ The semigroup $\PP$ is 
then defined by applying $\SS$ to each factor
\begin{align*}
 \PP(t)\sum_{\sigma\in S_n}(h_{\sigma(1)} \ot \ldots \ot h_{\sigma(n)}) := \sum_{\sigma\in S_n}\SS(t) h_{\sigma(1)} \ot \ldots \ot  \SS(t) h_{\sigma(n)},
\end{align*}
where $S_n$ is the permutation group on $\{1,\dots,n\}$.
For the details of this construction we refer to \cite{Janson}.
Equivalently, the semigroup $\PP$ can be defined via the the following generalisation of the classical Mehler formula,
\begin{align*}
 \PP(t) f(x) = \int_E f( \SS(t) x + \sqrt{I - \SS^*(t)\SS(t)} y ) \, d\mu(y),
\end{align*}
which makes sense by virtue of the fact that any bounded linear operator on $\HH$ admits a unique measurable linear extension to $E$ \cite{Bo}.
The generator $\LL$ of the semigroup $\PP$ is the elliptic operator formally given by 
\begin{align*}
 \LL = D^* \AA D,
\end{align*}
where $D$ denotes the Malliavin derivative associated with $\mu$ and its adjoint $D^*$ is the associated divergence operator.
The application to Ornstein-Uhlenbeck operators described in this paper is obtained by taking
$\mu \sim \mu_\infty$ and $\AA \sim A_\infty^*$ (cf. \cite{CG96, Ne98}).

\section{An example}

In this section we present an example of a Hilbert space $E$, a continuously
embedded Hilbert subspace
$H\embed E$, and a $C_0$-semigroup generator $A$ on $E$
such that:
\bit
\item the semigroup $S$ generated by $A$ fails to be analytic;
\item the stochastic Cauchy problem $$dU(t) = AU(t)\,dt + dW_H(t)$$ 
admits a unique invariant measure, which we denote by $\mu_\infty$;
\item the associated Ornstein-Uhlenbeck semigroup $P$ is analytic on $L^2(E,\mu_\infty).$
\eit
Thus, although analyticity of $P$ implies analyticity
of $S_H$ (Theorem \ref{thm:H-invar}), it does not imply analyticity of $S$.

Let $E = L^2(\R_+, e^{-x}\,dx)$ be the space of all measurable functions $f$
on $\R_+$ such that
$$ \n f\n := \Big(\int^{\infty}_0 |f(x)|^2\,e^{-x}dx \Big)^\frac12
< \infty.$$
The rescaled left translation semigroup $S$,
$$ S(t)f(x) := e^{-t} f(x+t), \quad f\in E, \ t>0, \ x >0,$$
is strongly continuous and contractive on $E$, and satisfies
$\n S(t)\n  = e^{-t/2}.$
Let
$H = H^2(\C_+)$ be the Hardy space of analytic functions $g$
on the open right-half plane $\C_+ = \{z\in \C: \ \Re z>0\}$
such that
$$ \n g\n_H := \sup_{x>0} \Big(\int_{-\infty}^\infty |g(x+iy)|^2\,dy\Big)^\frac12
< \infty.$$ Since $\lim_{x\to +\infty} g(x) = 0$ for all $g\in
H$, the restriction mapping $i: g\mapsto g|_{\R_+}$ is
well-defined as a bounded operator from $H$ to $E$. By
uniqueness of analytic continuation, this mapping is
injective. Since $i$ factors through $L^\infty(\R_+,
e^{-x}\,dx)$, $i$ is Hilbert-Schmidt \cite[Corollary 5.21]{ISEM}.
As a consequence (see,
e.g., \cite[Chapter 11]{DaPZab}), the Cauchy problem $ dU(t) =
AU(t)\,dt + \,dW_H(t)$ admits a unique invariant measure
$\mu_\infty$.

The rescaled left translation semigroup $S_H$,
$$ S_H(t)g(z) := e^{-t} g(z+t), \quad f\in H, \ t\ge 0, \ \Re z > 0,$$
is strongly continuous on $H$, it extends to an analytic
contraction semigroup of angle
$\frac12\pi$, and satisfies $\n S_H(t)\n_H = e^{-t/2}$.
Clearly, for all $t\ge 0$ we have
$ S(t)\circ i = i\circ S_H(t).$ By these observations combined with
\cite[Theorem 9.2]{GvN03}, the associated
Ornstein-Uhlenbeck semigroup $P$ is analytic.

\section{Application to the stochastic heat equation}

In this final section we shall apply our results to the following stochastic PDE 
with additive space-time white noise: 
\begin{equation}\label{eq:SH}
\begin{aligned} \frac{\partial u}{\partial t}(t,y) & = \frac{\partial^2 u}{\partial y^2}(t,y) + 
\frac{\partial^2 W}{\partial t\,  \partial y}(t,y), && t\ge 0, \ y\in [0,1], \\
u(t,0) & = u(t,1)  = 0, && t\ge 0,\\
u(0,y) & = 0, && y\in [0,1].
\end{aligned}
\end{equation}
This equation can be cast into the abstract form \eqref{SCP} by taking $H=E=L^2(0,1)$
and $A$ the Dirichlet Laplacian $\Delta$ on $E$.
The resulting equation 
$$
\begin{aligned} dU(t) & = A U(t)\,dt + dW(t), \\
U(0) & = 0,\end{aligned}
$$
where now $W$ denotes an $H$-cylindrical Brownian motion,
has a unique solution $U$ given by
$$ U(t) = \int_0^t S(t-s)\,dW(s), \quad t\ge 0,$$
where $S$ denotes the heat semigroup on $E$ generated by $A$.
Let $\mu_\infty$ denote the unique invariant measure on $E$ associated with $U$,
and let $H_\infty$ denote its reproducing kernel Hilbert space. Let
$i_\infty: H_\infty\embed E$ denote the canonical embedding and 
let $i:H\to E$ be the identity mapping. 
By \cite[Theorem 3.5, Corollary 5.6]{GGvN03} the densely defined operator
$V: i_\infty\s x\s \mapsto i\s x\s$ defined in \eqref{eq:defV}
is closable from $H_\infty$ to $H$.

Let $L$ be the generator of the Ornstein-Uhlenbeck semigroup $P$ on $L^p(E,\mu_\infty)$ 
associated with $U$. Since $P$ is analytic, the results of 
Sections \ref{sec:gradient} and 
\ref{sec:gradient-anal} can be applied. 
Noting that $\Delta$ is selfadjoint on $H,$
condition (3) of Theorem \ref{thm:kato} is satisfied and therefore
\begin{equation*} \Dom_p(\sqrt{-L}) = \Dom_p(D) \quad (1<p<\infty)
\end{equation*}
where $D = D_H = D_E$ denotes the Fr\'echet derivative on $L^p(E,\mu_\infty)$. 
 
One can go a step further by noting that the problem \eqref{eq:SH}
is well-posed even on the space $$\wt E := C_0[0,1] = \{f\in C[0,1]: \ f(0)=f(1) = 0\},$$
in the sense that the random variables $U(t)$ are $\wt E$-valued almost surely and that $U$
admits has a modification $\wt U$ with continuous (in fact, even H\"older continuous) trajectories in $\wt E$. 
Moreover, the invariant measure $\mu_\infty$ is supported on $\wt E$.
In analogy to \eqref{eq:P} this allows us to define an ``Ornstein-Uhlenbeck semigroup'' 
$\wt P$ on $L^p(\wt E,\mu_\infty)$ associated with $\wt U$ by
$$ \wt P(t)f(x):= \E f(\wt U^x(t)), \quad t\ge 0, \ x\in \wt E,$$
where $\wt U^x(t) = \wt S(t)x + \wt U(t)$ and $\wt S$ is the heat semigroup on $\wt E$. 
It is important to observe that we are not in the framework considered in the previous sections,
due to the fact that $H = L^2(0,1)$ is not continuously embedded in $\tilde E$.
Let $\wt L$ denote the generator of $\wt P$.
Under the natural identification 
$$L^p(\wt E,\mu_\infty) = L^p(E,\mu_\infty)$$ (using that the underlying measure spaces 
are identical up to a set of measure zero), we have $\wt P(t) = P(t)$ and $\wt L = L$, 
so that 
\begin{equation}\label{eq:domtL1}  
\Dom_p(\sqrt{-\wt L}) = \Dom_p(\sqrt{-L}) = \Dom_p(D)\quad (1<p<\infty).
\end{equation}

This representation may seem somewhat unsatisfactory, as the right-hand side refers explicitly to the ambient space $E$ in which $\wt E$ is embedded. 
An intrinsic representation of 
$  \Dom_p(\sqrt{-\wt L})$ can be obtained as follows.
For functions $F:\wt E\to \R$ of the form
$$
 F(f) = \phi\Big(\int_0^1 fg_1\,dt, \, \dots\, , \int_0^1 fg_N\,dt\Big), \quad f\in \wt E,
$$
with $\phi\in C_{\rm b}^2(\R^N)$ and $g_1,\dots,g_N\in H$, we define 
$\wt D F:\wt E\to H$ by
$$ \wt D F(f) = \sum_{n=1}^N \frac{\partial \phi}{\partial y_n}\Big(\int_0^1 fg_1\,dt, \, \dots\, , \int_0^1 fg_N\,dt\Big)g_n, 
\quad f\in\wt E.$$
This operator is closable in $L^p(\wt E,\mu_\infty)$ for all $1\le p<\infty$.
On $L^2(\wt E,\mu_\infty)$ we have the representation
$$\wt L = \wt D\s \wt D.$$ 
As a result we can apply
\cite[Theorem 2.1]{MN-kato} directly to the operator $V$ 
and obtain that
\begin{equation}\label{eq:domtL2}
  \Dom_p(\sqrt{-\wt L}) = \Dom_p(\wt D) \quad (1<p<\infty).
\end{equation}
This answers a question raised by Zdzis{\l}aw Brze\'zniak (personal communication).
To make the link between the formulas \eqref{eq:domtL1} and \eqref{eq:domtL2} note that, under the identification $L^p(\wt E,\mu_\infty) = L^p(E,\mu_\infty)$,
one also has $\Dom_p(\wt D) = \Dom_p(D)$.
 
\begin{rem}\label{rem:explicit}
It is possible to give explicit representations for the space $H_\infty$ 
and the operator $V$.
To begin with, the covariance operator $Q_\infty$ of $\mu_\infty$ 
is given by
$$ Q_\infty f = \int_0^\infty S(t)S\s(t)f\,dt = \int_0^\infty S(2t)f \,dt = \tfrac12\Delta^{-1} f, \quad f\in E.$$
It follows that the reproducing kernel Hilbert space $H_\infty$ associated with $\mu_\infty$ equals
\begin{align*}
H_\infty =  \Ran(\sqrt{Q_\infty}) =\Dom(\sqrt{-\Delta}) = H_0^1(0,1). 
\end{align*}
Noting that $Q_\infty = i_\infty\circ i_\infty\s$, we see that the operator
$V: i_\infty\s x\s \mapsto i\s x\s$ is given by 
\begin{align*} \Dom(V) &= H^2(0,1)\cap H_0^1(0,1), \\
 Vf &= 2\Delta f, \quad f\in \Dom(V).    
   \end{align*}
\end{rem}

\begin{rem}
Formulas for $\Dom_p(\wt L)$ analogous to \eqref{eq:domtL1} and \eqref{eq:domtL2} 
can be deduced from Theorem \ref{thm:kato} and \cite[Theorem 2.2]{MN-kato} 
in a similar way.
\end{rem}

The Ornstein-Uhlenbeck operators $L$ and $\wt L$ considered above are symmetric on 
$L^2(E,\mu_\infty)$, 
and therefore the domain identifications for their square roots could essentially 
be obtained from the results of \cite{CG01, Sh92}. The above argument, however, 
can be applied to a large class of second order elliptic differential operators 
$A$ on $L^2(0,1)$ (but explicit representations as in Remark \ref{rem:explicit} 
are only possible when $A$ is selfadjoint). 

In fact, under mild assumptions on the coefficients and under various types 
of boundary conditions, such operators $A$ have a bounded
$H^\infty$-calculus on $H=E=L^2(0,1)$ (see \cite{DDHPV, DuoMcI, KaKuWe} and there 
references therein). By the result of Le Merdy \cite{LeM-sim} mentioned earlier, 
this implies that the analytic semigroup $S$ generated 
by $A$ is contractive in some equivalent Hilbertian norm. Hence, by \cite[Theorem 9.2]{GvN03},
the associated Ornstein-Uhlenbeck semigroup is analytic.
Typically, under Dirichlet boundary conditions, $S$ is uniformly exponentially stable. 
This implies (see \cite{DaPZab}) that the solution $U$ of 
\eqref{SCP} admits a unique invariant measure. Finally, the analyticity of $S$
typically implies space-time H\"older regularity of $U$ (see \cite{Brz97, DNW}), so that the corresponding 
stochastic PDE is 
well-posed in $\wt E = C_0[0,1]$. We plan to provide more details in a 
forthcoming publication.

\bibliographystyle{ams-pln}
\bibliography{Maas_vanNeerven}

\end{document}